\documentclass[12pt,reqno]{amsart}
\usepackage{amsmath, amssymb, amsthm} 
\usepackage{mathrsfs}
\usepackage[breaklinks]{hyperref}
\usepackage{tikz}
\usepackage{mathrsfs}
\usepackage{latexsym}
\usepackage{upref, eucal}
\usepackage[headheight=110pt,top=1in, bottom=.9in, left=0.8in, right=0.8in]{geometry}

\usepackage{url}
\usepackage{amssymb}

\newcommand{\df}{\dfrac}
\newcommand{\tf}{\tfrac}

\newcommand{\g}{\gamma}

\renewcommand{\(}{\left\(}
\renewcommand{\)}{\right\)}
\renewcommand{\[}{\left\[}
\renewcommand{\]}{\right\]}
\renewcommand{\i}{\infty}
\numberwithin{equation}{section}
\theoremstyle{plain}
\newtheorem{theorem}{Theorem}[section]
\newtheorem{lemma}[theorem]{Lemma}

\makeatletter
\def\proof{\@ifnextchar[{\@oproof}{\@nproof}}
\def\@oproof[#1][#2]{\trivlist\item[\hskip\labelsep\textit{#2 Proof of\
		#1.}~]\ignorespaces}
\def\@nproof{\trivlist\item[\hskip\labelsep\textit{Proof.}~]\ignorespaces}

\makeatother

\makeatletter
\def\@tocline#1#2#3#4#5#6#7{\relax
	\ifnum #1>\c@tocdepth 
	\else
	\par \addpenalty\@secpenalty\addvspace{#2}%
	\begingroup \hyphenpenalty\@M
	\@ifempty{#4}{%
		\@tempdima\csname r@tocindent\number#1\endcsname\relax
	}{%
		\@tempdima#4\relax
	}%
	\parindent\z@ \leftskip#3\relax \advance\leftskip\@tempdima\relax
	\rightskip\@pnumwidth plus4em \parfillskip-\@pnumwidth
	#5\leavevmode\hskip-\@tempdima
	\ifcase #1
	\or\or \hskip 1em \or \hskip 2em \else \hskip 3em \fi%
	#6\nobreak\relax
	\dotfill\hbox to\@pnumwidth{\@tocpagenum{#7}}\par
	\nobreak
	\endgroup
	\fi}
\makeatother

\allowdisplaybreaks

\begin{document}
	\title[On a function of Ramanujan  twisted by a logarithm]{On a function of Ramanujan  twisted by a logarithm}
	\author{Atul Dixit}
	\address{Department of Mathematics, Indian Institute of Technology Gandhinagar, Palaj, Gandhinagar, Gujarat-382355, India}
	\email{adixit@iitgn.ac.in}
	\author{Sumukha Sathyanarayana}
	\address{Department of Mathematics, Indian Institute of Technology Gandhinagar, Palaj, Gandhinagar, Gujarat-382355, India}
	\email{sumukhas@cuk.ac.in}
	\curraddr{Department of Mathematics, Central University of Karnataka, Kadaganchi, Kalaburagi, Karnataka 585367}
	\author{N. Guru Sharan }
	\address{Department of Mathematics, Indian Institute of Technology Gandhinagar, Palaj, Gandhinagar, Gujarat-382355, India}
	\email{gurusharan.n@iitgn.ac.in}
		\maketitle
	\begin{center}
		\dedicatory{\emph{Dedicated to Professors George Andrews and Bruce Berndt for their 85th birthdays}}
	\end{center}

	\begin{abstract}
		A two-term functional equation for an infinite series involving the digamma function and a logarithmic factor is derived. A modular relation on page 220 of Ramanujan's Lost Notebook as well as a corresponding recent result for the derivative of Deninger's function are two main ingredients in its derivation. An interesting  integral $\mathscr{H}(x)$, which is of independent interest, plays a prominent role in our functional equation. Several alternative representations for $\mathscr{H}(x)$ are obtained.
	\end{abstract}

\section{Introduction}\label{intro}
A Kronecker limit formula for a zeta function concerns evaluating the constant term in the Laurent series expansion around its pole, the simplest of which is the one for the Riemann zeta function determined by
\begin{equation}\label{zeta laurent}
	\zeta(s)=\frac{1}{s-1}+\gamma+\sum_{j=1}^{\infty}\frac{(-1)^j}{j!}\gamma_j(s-1)^j,
\end{equation} 
where, $\gamma_j, j\in\mathbb{N}\cup\{0\}$, are \emph{Stieltjes constants}, and $\gamma_0:=\gamma$ is the Euler constant. 

Kronecker limit formulas have been obtained for various zeta functions. Hecke was the first mathematician to derive it for the zeta function associated with an imaginary quadratic field. While he also obtained the corresponding formula for a real quadratic field, his result was not very explicit as it involved the logarithm of the Dedekind eta function $\eta(z)$. Later, Zagier \cite{zagier} obtained an explicit formula for real quadratic fields. Applications of Kronecker limit formulas include the solution to Pell's equation (the construction of a fundamental unit of a real quadratic field), finding class numbers of number fields, determining Hecke $L$-values etc.

Zagier's version of the Kronecker limit formula for a real quadratic field involved a curious analytic function $F(x)$, now known as the \emph{Herglotz-Zagier} function. For $x\in\mathbb{C}\backslash(-\infty,0]$, it is defined by
\begin{equation}\label{herglotzdef}
	F(x):=\sum_{n=1}^{\infty}\frac{\psi(nx)-\log(nx)}{n},
\end{equation} 
where $\psi(x):=\Gamma'(x)/\Gamma(x)$ is the logarithmic derivative of the Gamma function $\Gamma(x)$, also known as the digamma function. Since
\cite[p.~259, formula 6.3.18]{as}, for $|\arg$ $z|<\pi$, 
\begin{equation}\label{psiasymp}
	\psi(z) \sim \log z-\frac{1}{2z}-\frac{1}{12z^{2}}+\frac{1}{120z^{4}}-\frac{1}{252z^{6}}+\cdots,
\end{equation}
as
$z\to\infty$, we see that the series in \eqref{herglotzdef} converges absolutely.

The function $F(x)$ is interesting in its own right. It satisfies beautiful two- and three-term functional equations,
due to Zagier \cite[Equations (7.4), (7.8)]{zagier}, namely,
\begin{align*}
	&F(x)+F\left(\frac{1}{x}\right)=2F(1)+\frac{1}{2}\log^{2}(x)-\frac{\pi^2}{6x}(x-1)^2,\\
	&F(x)-F(x+1)-F\left(\frac{x}{x+1}\right)=-F(1)+\textup{Li}_2\left(\frac{1}{1+x}\right),
\end{align*}
with $F(1)$ being \cite[Equation (7.12)]{zagier}
\begin{equation*}
	F(1)=-\frac{1}{2}\g^2-\frac{\pi^2}{12}-\g_1,
\end{equation*}
and $\textup{Li}_{2}(z)$ is the well-known dilogarithm function.

The literature on Herglotz function and its various generalizations and analogues is vast; see \cite{kumar-choie}, \cite{hhf1}, \cite{dss2}, \cite{ishibashi}, \cite{masri}, \cite{raza}, and \cite{vz}.

On page 220 of the Lost Notebook \cite{lnb}, Ramanujan gave a beautiful two-term functional equation for a Herglotz-Zagier type function.
For $x\in\mathbb{C}\backslash(-\infty,0]$, let
\begin{align}
	\phi_0(x):=\psi(x)+\frac{1}{2x}-\log(x). \label{phi0 defn}
\end{align}
Then for Re$(\alpha)>, \textup{Re}(\beta)>0$ such that $\alpha \beta=1$, then
\begin{align}\label{w1.26}
	\sqrt{\alpha}\left\{ \sum_{n=1}^{\i}\phi_0(n\alpha)+ \frac{\gamma-\log(2\pi\alpha)}{2\alpha}\right\}
	&=\sqrt{\beta}\left\{\sum_{n=1}^{\i}\phi_0(n\beta)+\frac{\gamma-\log(2\pi\beta)}{2\beta}\right\}\nonumber\\
	&=-\df{1}{\pi^{3/2}}\int_0^{\i}\left|\Xi\left(\df{1}{2}t\right)\Gamma\left(\df{-1+it}{4}\right)\right|^2
	\df{\cos\left(\tf{1}{2}t\log\alpha\right)}{1+t^2}\, dt,
\end{align}
where $\Xi$ and $\xi$ are the Riemann's functions defined by
\begin{align*}
	\Xi(t)&:=\xi\left(\frac{1}{2}+it\right), \\
	\xi(s)&:=\frac{1}{2}s(s-1)\pi^{-\frac{s}{2}}\Gamma\left(\frac{s}{2}\right)\zeta(s).
\end{align*}
A two-term functional equation of the type $F(\alpha)=F(\beta)$, where $\alpha\beta=1$ with Re$(\alpha)>0$ and Re$(\beta)>0$, such as the one in \eqref{w1.26}, is also sometimes called a \emph{modular relation}.

The series and the integral occurring in identity \eqref{w1.26} satisfy striking properties and have nice applications. We now illustrate some of them. The first equality in \eqref{w1.26} can be used to obtain an asymptotic expansion of the series $\sum_{n=1}^{\infty}\phi_0(n\alpha)$ as $\alpha\to0$. See Oloa \cite[Section 4]{oloa}.
Regarding the integral in \eqref{w1.26}, Hardy \cite{ghh} said \textit{``the properties of this integral resemble those of one which Mr. Littlewood and I have used, in a paper to be published shortly in Acta Mathematica to prove that\footnote{There is a typo in this formula in that $\pi$ should not be present.}}
\begin{equation*}
	\int_{-T}^{T}\left|\zeta\left(\frac{1}{2}+ti\right)\right|^2\, dt \sim
	\frac{2}{\pi} T\log T\hspace{3mm}(T\to\infty)\textup{''}.
\end{equation*}
A remarkable Fourier transform equal to the integral in \eqref{w1.26}, also, given by Ramanujan \cite[Equation (22)]{riemann}, was recently used by Darses and Najnudel \cite{darses-najnudel} to obtain exact formulas for weighted $2k^{\textup{th}}$ moments of the Riemann zeta function, where $k\in\mathbb{N}$, in terms of an auto-correlation function $A(z)$ defined for $\textup{Re}(z)>0$ by
\begin{equation}\label{auto-correlation function}
A(z):=\int_{0}^{\infty}\left(\frac{1}{xz}-\frac{1}{e^{xz}-1}\right)	\left(\frac{1}{x}-\frac{1}{e^{x}-1}\right)\, dx.
\end{equation}
For $x\in\mathbb{C}\backslash(-\infty,0]$, consider Ramanujan's function $\sum_{n=1}^{\infty}\phi_0(nx)$ twisted by a logarithm, that is,
\begin{equation}\label{G of x}
	\phi_{\log}(x):=\sum_{n=1}^{\infty}\left(\psi(nx)+\frac{1}{2nx}-\log(nx)\right)\log(n\sqrt{x}).
\end{equation}
Using \eqref{psiasymp}, one can see that the above series converges absolutely for $x\in\mathbb{C}\backslash(-\infty,0]$.
 
The main result of our paper shows that $\phi_{\log}(x)$
satisfies a nice two-term functional equation. It also involves an integral $\mathscr{H}(x)$ defined for $x>0$ and $0<c<1$ by
\begin{align}
	\mathscr{H}(x)&:= \frac{1}{2 \pi i} \int_{(c)} \frac{\pi \zeta(w)  \zeta(1-w)}{\sin(\pi w)} \left( \pi \cot\left(\frac{\pi w}{2}\right) \right) x^{-w} \, dw, \label{H defn} 
\end{align}
where, here, and in the sequel, we have used the notation $\int_{(c)}$ to denote the line integral $\int_{c-i\infty}^{c+i\infty}$. Even though a closed-form evaluation of $\mathscr{H}(x)$ looks difficult, we derive several series or integral representations for it in Section \ref{repH}, namely, \eqref{1identity}, \eqref{2identity} and \eqref{2.5identity}.

Our main result is as follows.
\begin{theorem}\label{log-twist}
Let $\phi_{\log}(x)$ and $\mathscr{H}(x)$ be defined in \eqref{G of x} and \eqref{H defn} respectively. Define
$\mathscr{G}(x)$ by
\begin{align*}
	\mathscr{G}(x):=  \phi_{\log}(x)- \frac{1}{2} \mathscr{H}(x) -\frac{1}{48 x} ( 12\gamma^2 - 5\pi^2)   + \frac{1}{4x} \left(  \gamma\log(x) +\log(2\pi) \log(2\pi x) \right).
\end{align*}
Then, for any $\alpha, \beta>0$ such that $\alpha\beta=1$,
\begin{align}\label{modular log-twist}
	\sqrt{\alpha}\mathscr{G}(\alpha)=\sqrt{\beta}\mathscr{G}(\beta).
\end{align}
\end{theorem}
Next, we try to convince the reader why the result in \eqref{modular log-twist} is unexpected and surprising. 

Maier \cite{maier} was the first mathematician to examine the higher Herglotz function $\sum_{n=1}^{\infty}\psi(nx)n^{-s}$
 for $|\arg(x)|<\pi$ and a continuous variable $s$ such that $\textup{Re}(s)>1$, and to derive, for positive integers $s>1$, its two-term functional equation, later rediscovered by Vlasenko and Zagier see \cite[p.~114, Equation (3)]{maier}, \cite[Equation (11)]{vz}. 
 Observe that
 \begin{equation}\label{derivative}
 	\frac{d}{ds}\sum_{n=1}^{\infty}\left.\frac{1}{n^s}\left(\psi(nx)+\frac{1}{2nx}-\log(n)\right)\right|_{s=0}=-\sum_{n=1}^{\infty}\left(\psi(nx)+\frac{1}{2nx}-\log(nx)\right)\log(n).
\end{equation}
There are no known two-term functional equations for the series $\sum_{n=1}^{\infty}n^{-s}\left(\psi(nx)+\frac{1}{2nx}-\log(n)\right)$ for a \emph{complex} $s$, however, and so the fact there still exists a functional equation for the series on the right-hand side of \eqref{derivative} makes it interesting.  
 
We note in passing that for real numbers $k$ and $N$ such that $k+2N>1$, Gupta and Kumar \cite[Equation (1.9)]{gupta-kumar-ehhf} studied the extended higher Herglotz function
 \begin{align*}
 \mathscr{F}_{k, N}(x):=\sum_{n=1}^{\infty}\frac{1}{n^k}\left(\psi(n^{N}x)+\frac{1}{2n^Nx}-\log(n^Nx)\right),
 \end{align*}
from the point of view of deriving functional equations similar to \eqref{w1.26}, however, the power of $n$ in the denominator there, that is, $k$, is still not a complex number so that differentiation does not make sense.
 
 The proof of Theorem \ref{log-twist} first requires an intermediate modular relation to be derived. This result, given in Theorem \ref{maintheorem} below, contains infinite series whose summands involve the function $\psi_1(x)$. This function is a special case of the generalized digamma function $\psi_k(x)$ defined for $x\in\mathbb{C}\backslash(-\infty,0]$ by
 \begin{equation*}
 	\psi_k(x):=\frac{\Gamma_k'(x)}{\Gamma_k(x)}=-\gamma_k-\frac{\log^{k}(x)}{x}-\sum_{n=1}^{\infty}\left(\frac{\log^{k}(n+x)}{n+x}-\frac{\log^{k}(n)}{n}\right),
 \end{equation*}
where $\Gamma_k(n)$ is the generalized gamma function studied by Dilcher \cite{dilcher}, namely,
	\begin{equation*}
	\Gamma_k(x):=\lim_{n\to\infty}\frac{\text{exp}{\left(\frac{\log^{k+1}(n)}{k+1}x\right)}\prod\limits_{j=1}^{n}\text{exp}{\left(\frac{\log^{k+1}(j)}{k+1}\right)}}{\prod\limits_{j=0}^{n}\text{exp}{\left(\frac{\log^{k+1}(j+x)}{k+1}\right)}}.
\end{equation*}
Observe that $\Gamma_0(x)=\Gamma(x)$ and $\psi_0(x)=\psi(x)$.
The function $\psi_k(x)$ was briefly treated by Ramanujan in Entry 22 of Chapter 8 of his second notebook (see \cite{RN_1}) whereas Deninger studied the functions $\log(\Gamma_1(x))$ and $\psi_1(x)$  during his investigation of an analogue of the Chowla-Selberg formula for real quadratic fields. The function $\log(\Gamma_1(x))$ is now known as \emph{Deninger's function}.

One of the ingredients in the proof of Theorem \ref{log-twist} is an analogue of \eqref{w1.26} for $\psi_1(x)$ obtained by the present authors in \cite{dss}. It is stated next. Define 
\begin{align}
	\phi_1(x)&:= \psi_1(x) + \frac{1}{2x}\log(x)- \frac{1}{2}\log^2(x), \label{phi1 defn}
\end{align}
and
\small\begin{align}\label{ef1}
	\mathscr{F}_1(x)&:=\sqrt{x} \left\{ \sum_{n=1}^{\infty} \phi_1(nx) + \frac{\log^2(2\pi)-(\gamma-\log(x))^2}{4x} + \frac{\pi^2}{48x} \right\} -\frac{\sqrt{x} \log(x)}{2} \left\{ \sum_{n=1}^{\infty} \phi_0(nx) +\frac{\gamma - \log (2\pi x)}{2x} \right\}.
\end{align}
\normalsize
Then, for any $\alpha,\beta $ such that \textup{Re}$(\alpha)>0$, \textup{Re}$(\beta)>0$ and $\alpha \beta =1$, we have \cite[Theorem 2.2]{dss}\footnote{In the present paper, we have only used the first equality of \cite[Theorem 2.2]{dss} as the integral involving the Riemann $\Xi$-function does not have a role to play here.}
\begin{align}\label{k=1}
	\mathscr{F}_1(\alpha)=	\mathscr{F}_1(\beta).
\end{align}
We now explain why the series in \eqref{ef1} converges absolutely for any $x>0$. The Laurent series expansion of the Hurwitz zeta function $\zeta(z, x)$ around $z=1$ is given by 
\begin{equation*}
	\zeta( z, x)=\frac{1}{z-1}+\sum_{j=0}^{\infty}\frac{(-1)^j\gamma_j(x)}{j!}(z-1)^j.
\end{equation*}
Note that when $x=1$, this reduces to \eqref{zeta laurent}. For any $j\in\mathbb{N}\cup\{0\}$, Shirasaka \cite{shirasaka} proved that $\psi_j(x)$ is the additive inverse of the generalized Stieltjes constants $\gamma_j(x)$. 
 This, along with the asymptotic expansion of $\gamma_j(x)$ given in Proposition 3 of \cite{coffey}, gives us the following asymptotic expansion  as $x \to \infty$,  Re$(x)>0$:
\begin{align*}
	\psi_j(x)\sim\frac{\log^{j+1}(x)}{j+1}-\frac{1}{2x}\log^{j}(x)+\sum_{m=1}^{\infty}\frac{B_{2m}}{(2m)!}x^{-2m}\sum_{t=0}^{j}\binom{j}{t}t!s(2m,t+1)\log^{j-t}(x),
\end{align*}
where $B_m$ denotes $m^{\textup{th}}$ Bernoulli number. This clearly shows that the series $\sum_{n=1}^{\infty}\phi_1(nx)$ converges absolutely.

\section{Preliminaries}
In this section we record a few well-known results which are used in the sequel. We begin with the asymmetric version of the functional equation of $\zeta(w)$, for any $w \in \mathbb{C}$ \cite[p.~259, Theorem 12.7]{Apostol}: 
\begin{align}
	\zeta(1-w)= 2 (2\pi)^{-w} \Gamma(w) \zeta(w)\cos\left( \frac{\pi w}{2} \right). \label{asymmetric fe of zeta} 
\end{align}
Next, Stirling's formula for Gamma function in a vertical strip $C\leq\sigma\leq D$ is given by \cite[p.224]{cop}
\begin{equation}\label{Stirling_equn}
	|\Gamma (\sigma + i T)| = \sqrt{2\pi} | T|^{\sigma - 1/2} e^{-\frac{1}{2} \pi |T|} \left(1 + O\left(\frac{1}{|T|}\right)  \right) \quad {\rm as} \quad |T|\rightarrow \infty.
\end{equation}

We now state two versions of \emph{Parseval's formula} to be used in Section \ref{repH}. 
\begin{theorem}\textup{\cite[p. 83, Equation (3.1.11)]{pariskaminski}} 	
	Let $H(s)$ and $G(s)$ be the Mellin transforms of $h(x)$ and $g(x)$ respectively. Suppose that the conditions
		$x^{d-1}g(x) \in L [0, \infty )$ and $H(1-d-ix) \in  L(-\infty, \infty )$
	hold.
	If $H(1-s)$ and $G(s)$ have a common strip of analyticity,  then for any vertical line $\textup{Re}(s)=d$ in the common strip, we have
	\begin{align}\label{Par1}
		\frac{1}{2 \pi i } \int_{(d)} H(1-w) G(w) \, dw =\int_{0}^{\infty}  h(x) g(x) dx,
	\end{align}
	under the assumption that the integral on the right-hand side exists.
\end{theorem}
Another useful version of this result is
\cite[p.~83, Equation (3.1.13)]{pariskaminski}, if $H(s)$ and $G(s)$ have a common strip of analyticity,  then for any vertical line Re$(s)=d$ in the common strip,
	\begin{align}\label{Par2}
		\frac{1}{2 \pi i } \int_{(d)} H(w) G(w) u^{-w} \, dw =\int_{0}^{\infty}  \frac{1}{x} h\left(\frac{u}{x}\right) g(x) dx.
	\end{align}
\section{Proof of the main result}
The proof of Theorem \ref{log-twist} relies on an intermediate theorem, namely, Theorem \ref{maintheorem}. 
The proof of the latter, in turn, involves a modular relation for an integral $\mathscr{J}(x)$ defined for $x>0$ and $0<c<1$ by
\begin{align}
	\mathscr{J}(x)&:= \frac{1}{2 \pi i} \int_{(c)} \frac{\pi \zeta(w)  \zeta(1-w)}{\sin(\pi w)} \left( \psi(w)  + \frac{\pi}{2} \cot\left(\frac{\pi w}{2}\right) \right) {x}^{-w} \, dw, \label{J defn}
\end{align}
This is exactly what we now set to prove.
	\begin{lemma}\label{J lemma}
	Let $\alpha, \beta>0$ such that $\alpha\beta=1$, then
	\begin{align}
		\sqrt{\alpha} \mathscr{J}(\alpha)=\sqrt{\beta} \mathscr{J}(\beta) \label{Jequation}.
	\end{align}
\end{lemma}
\begin{proof}
	Replacing $w$ by $1-w$ in the definition of $\mathscr{J}(\beta)$ and using \cite[p.~76, Exercise \textbf{3.17}]{temme} 
	 \begin{align}\label{psi reflection}
	 \psi(1-w)=\psi(w)+ \pi \cot(\pi w)
	\end{align} 
 in the second step, we see that
	\begin{align*}
		\beta	\mathscr{J}(\beta)&=   \frac{1}{2 \pi i}  \int_{(c)} \frac{\pi \zeta(w)  \zeta(1-w)}{\sin(\pi w)} \left( \psi(1-w) + \frac{\pi}{2} \cot\left(\frac{\pi (1-w)}{2}\right) \right) {\beta}^{w} \, dw \\
		&=    \frac{1}{2 \pi i}  \int_{(c)} \frac{\pi \zeta(w)  \zeta(1-w)}{\sin(\pi w)} \left( \psi(w) + \pi \cot(\pi w) + \frac{\pi}{2} \tan\left(\frac{\pi w}{2}\right) \right) {\alpha}^{-w} \, dw \\
		&=     \frac{1}{2 \pi i}  \int_{(c)} \frac{\pi \zeta(w)  \zeta(1-w)}{\sin(\pi w)} \left( \psi(w) + \frac{\pi}{2} \cot\left(\frac{\pi w}{2}\right) \right) {\alpha}^{-w} \, dw \\
		&= \mathscr{J}(\alpha),
	\end{align*}
	where the second-to-last step follows from the fact that $\cot(\theta)-\tan(\theta)=2\cot(2\theta)$.
	\end{proof}
We now prove the intermediate theorem which will be used, in turn, to prove Theorem \ref{log-twist}.
\begin{theorem}\label{maintheorem}
	For $x>0$, define $\mathscr{F}(x)$ by
	\begin{align*}
		\mathscr{F}(x):=\sum_{n=1}^{\infty} \phi_1(nx)  - \sum_{n=1}^{\infty} \log(nx)\phi_0(nx)  +  \frac{1}{2} \mathscr{H}(x) - \frac{\pi^2}{12x}, \notag 
	\end{align*}
	where $\phi_0, \phi_1$ and $\mathscr{H}$ are as defined above in \eqref{phi0 defn}, \eqref{phi1 defn}, and \eqref{H defn} respectively. Then, for any $\alpha, \beta>0$ such that $\alpha\beta=1$,
	\begin{align}
		\sqrt{\alpha}\mathscr{F}(\alpha)=\sqrt{\beta}\mathscr{F}(\beta). \label{main result 1}
	\end{align}
\end{theorem}
\begin{proof}
From \cite[Theorem 3.3]{bdg_log},
\begin{align*}
	\frac{1}{2\pi i}\int_{c-i\infty}^{c+i\infty}\frac{\pi\zeta(1-s)}{\sin(\pi s)}\left(\gamma-\log(x)+\psi(s)\right)x^{-s}\, ds = \psi_1(x+1)-\frac{1}{2}\log^2(x),
\end{align*}
where $0<c=\textup{Re}(s)<1$. Shifting the line of integration from Re$(s)=c$ to Re$(s)=d$, where $1<d<2$ and noting the integrals along the horizontal segments go to zero (due to \eqref{Stirling_equn}) as the height of the contour tends to $\infty$, by Cauchy's residue theorem,
\begin{align*}
 \frac{1}{2 \pi i} \int_{(d)} \frac{\pi \zeta(1-s)}{\sin(\pi s)} (\gamma - \log(x) + \psi(s)) x^{-s} \, ds=	\phi_1(x), \label{psi1 mellin}
\end{align*}
where we can see that the residue of the integrand at $s=1$ encountered is $\log(x)/(2x)$, and where we used the functional equation \cite[Equation (1.9)]{dss}
\begin{equation*}
 \psi_1(x+1)=\psi_1(x)+\frac{\log(x)}{x}.
\end{equation*}
  Thus,
\begin{align*}
	\sqrt{\alpha}\sum_{n=1}^{\infty} \phi_1(n\alpha) =&\sqrt{\alpha} \sum_{n=1}^{\infty} \frac{1}{2 \pi i} \int_{(d)} \frac{\pi \zeta(1-s)}{\sin(\pi s)} (\gamma - \log(n\alpha) + \psi(s)) (n\alpha)^{-s} \, ds \\
	=& \frac{\sqrt{\alpha}}{2 \pi i} \int_{(d)} \frac{\pi \zeta(1-s)}{\sin(\pi s)} \left((\gamma - \log(\alpha) + \psi(s)) \zeta(s) + \zeta'(s)\right) {\alpha}^{-s} \, ds,
\end{align*}
where the interchange of the order of summation and integration is easily justified owing to the exponential decay of the integrand. We also used the fact $\sum_{n=1}^{\infty}\log(n)n^{-s}=-\zeta'(s)$ here. A shift of the line of integration back to $0<c<1$ and another application of the residue theorem yields
\begin{align}
	\sqrt{\alpha} \sum_{n=1}^{\infty} \phi_1(n\alpha)
	=& \frac{\sqrt{\alpha}}{2 \pi i} \int_{(c)} \frac{\pi \zeta(1-s)}{\sin(\pi s)} \left((\gamma - \log(\alpha) + \psi(s)) \zeta(s) + \zeta'(s)\right) {\alpha}^{-s} \, ds + \sqrt{\alpha} R_0(\alpha), \notag \\
	=& I(\alpha) + \sqrt{\alpha} R_0(\alpha).\label{sum integral R0}
\end{align}
where,
\begin{align}
	I(\alpha):= \frac{\sqrt{\alpha}}{2 \pi i} \int_{(c)} \frac{\pi \zeta(1-s)}{\sin(\pi s)} \left((\gamma - \log(\alpha) + \psi(s)) \zeta(s) + \zeta'(s)\right) {\alpha}^{-s} \, ds, \label{I defn}
\end{align}
and 	$R_0(\alpha)$  is the residue of the integrand at $s=1$ given by
\begin{align}
	R_0(\alpha):= \frac{1}{\alpha} \left( \frac{\gamma^2}{4} - \frac{\pi^2}{48} -\frac{1}{4} (\log^2 (2)+\log(\pi) \log(4 \pi))+ \frac{\log(\alpha)}{4} (-2 \gamma + \log(\alpha))\right).\label{R0 defn}
\end{align}
Employing the change of variable $s=1-w$ in the definition of $I(\alpha)$ in \eqref{I defn} and noting that $\alpha\beta=1$ leads to
\begin{align}
		I(\alpha) &= \frac{\sqrt{\alpha}}{2 \pi i} \int_{(c)} \frac{\pi \zeta(1-s) \zeta(s)}{\sin(\pi s)} \left(  \gamma - \log(\alpha) + \psi(s)  + \frac{\zeta'(s)}{\zeta(s)} \right) {\alpha}^{-s} \, ds \label{I integral} \\
		&=  \frac{\sqrt{\beta}}{2 \pi i} \int_{(c)} \frac{\pi \zeta(w) \zeta(1-w)}{\sin(\pi w)} \left( \gamma + \log(\beta) + \psi(1-w) + \frac{\zeta'(1-w)}{\zeta(1-w)} \right) {\beta}^{-w} \, dw.\nonumber
\end{align}
Using the formula \cite[p.~20]{titch}
\begin{align*}
\frac{\zeta'(1-w)}{\zeta(1-w)}=\log(2\pi)+\frac{\pi}{2}\tan\left(\frac{\pi w}{2}\right)-\psi(w)-\frac{\zeta'(w)}{\zeta(w)},	
\end{align*}
we write $I(\alpha)$ in the form
\begin{align}\label{I A1 A2}
	I(\alpha)=\sqrt{\beta} (A_{1}(\beta) + A_{2}(\beta)),
\end{align}
where
\begin{align}
	A_{1}(\beta)&= \frac{1}{2 \pi i} \int_{(c)} \frac{\pi \zeta(w)  \zeta(1-w)}{\sin(\pi w)} \bigg( \gamma - \log(\beta) + \psi(w)  + \frac{\zeta'(w)}{\zeta(w)} \bigg) {\beta}^{-w} \, dw ,\label{A1 defn} \\
	A_{2}(\beta)&= \frac{1}{2 \pi i} \int_{(c)} \frac{\pi \zeta(w)  \zeta(1-w)}{\sin(\pi w)} \bigg( \log(2\pi\beta^{2})  -2 \frac{\zeta'(w)}{\zeta(w)} + \psi(1-w) - 2\psi(w)  +\frac{\pi}{2} \tan\left(\frac{\pi w}{2}\right) \bigg) {\beta}^{-w} \, dw.\label{A2 defn}
\end{align}
Now use the reflection formula \eqref{psi reflection} and then the identity $\cot(\pi w/2)-\tan(\pi w/2)=2\cot(\pi w)$ in the definition of $A_2(\beta)$ to get
	\begin{align}
	A_{2}(\beta)&= \frac{1}{2 \pi i} \int_{(c)} \frac{\pi \zeta(w)  \zeta(1-w)}{\sin(\pi w)} \left(  \log(2\pi\beta^{2})  -2 \frac{\zeta'(w)}{\zeta(w)} - \psi(w) +\frac{\pi}{2} \cot\left(\frac{\pi w}{2}\right) \right) {\beta}^{-w} \, dw \notag \\
	&=  J_1(\beta) + J_2(\beta) - \mathscr{J}(\beta) + \mathscr{H}(\beta), \label{A2 J1 J2 J H}
	\end{align}
	where $\mathscr{J}$ and $\mathscr{H}$ are as defined in \eqref{J defn} and \eqref{H defn} respectively, and
	\begin{align}
		J_1(\beta)&:= \log(2\pi \beta^2)\cdot \frac{1}{2 \pi i} \int_{(c)} \frac{\pi \zeta(w)  \zeta(1-w)}{\sin(\pi w)}  {\beta}^{-w} \, dw, \label{J1 defn} \\
		J_2(\beta)&:= -2 \cdot \frac{1}{2 \pi i} \int_{(c)} \frac{\pi \zeta'(w)  \zeta(1-w)}{\sin(\pi w)}  {\beta}^{-w} \, dw. \label{J2 defn}
	\end{align}
	We first evaluate $J_1(\beta)$.
		In \eqref{J1 defn}, move the line of integration from Re$
		(s)=c$ to Re$(s)=d$ with $1<d<2$, and then use the series definition of $\zeta(w)$ in the subsequent step to arrive at
	\begin{align*}
		J_1(\beta) &= \log(2\pi \beta^2) \left(\frac{1}{2 \pi i} \int_{(d)} \frac{\pi \zeta(w)  \zeta(1-w)}{\sin(\pi w)} {\beta}^{-w} \, dw - \left( \frac{\gamma - \log(2 \pi \beta)}{2\beta} 	\right) \right)\\
		&= \log(2\pi \beta^2) \left( \sum_{n=1}^{\infty} \frac{1}{2 \pi i} \int_{(d)} \frac{\pi  \zeta(1-w)}{\sin(\pi w)} {(n\beta)}^{-w} \, dw - \left( \frac{\gamma - \log(2 \pi \beta)}{2\beta} 	\right) \right).
	\end{align*}
	We now move the line of integration from Re$
	(s)=d$ to Re$(s)=c$, where $0<c<1$, so as to get
	\begin{align*}
		J_1(\beta)
		= -\log(2\pi \beta^2) \left( \sum_{n=1}^{\infty} \left(  \frac{1}{2 \pi i} \int_{(c)} \left(-\frac{\pi  \zeta(1-w)}{\sin(\pi w)} {(n\beta)}^{-w} \right) \, dw  -\frac{1}{2n\beta}	 \right) + \left( \frac{\gamma - \log(2 \pi \beta)}{2\beta} 	\right) \right).
	\end{align*}
	Now employ Kloosterman's result \cite[p.34, equation  2.15.8]{titch}
	\begin{align}\label{P4}
		\frac{1}{2 \pi i } \int_{(c)} \frac{-\pi \zeta(1-w)}{\sin(\pi w)}  {x}^{-w} \, dw = \psi(x+1) -\log(x)
	\end{align}
 and the functional equation
 \begin{align}\label{psi fe}
  \psi(x+1)=\psi(x)+\frac{1}{x}
  \end{align}
  	to conclude that
		\begin{align}
		J_1(\beta)
		&= -\log(2\pi \beta^2) \left( \sum_{n=1}^{\infty} \phi_0(n\beta) + \left( \frac{\gamma - \log(2 \pi \beta)}{2\beta} 	\right) \right),\label{J1 last}
	\end{align}
	where $\phi_0$ is as defined in \eqref{phi0 defn}.
	
	We now handle $J_2(\beta)$ similarly. A shift of the line of integration from Re$
	(s)=c$ to Re$(s)=d$ with $1<d<2$, and the use the series definition of $\zeta'(w)$ produces
	\begin{align*}
		J_2(\beta)&= -2 \sum_{n=1}^{\infty} \log(n) \left( \frac{1}{2 \pi i} \int_{(d)} \frac{-\pi  \zeta(1-w)}{\sin(\pi w)} {(n\beta)}^{-w} \, dw \right) + 2 P_{1}(\beta),
	\end{align*}
	where
	\begin{align}
		P_1(\beta)=\frac{1}{\beta} \left( \frac{\gamma^2}{4} - \frac{5\pi^2}{48} -\frac{1}{4} \left( \log^2(2) + \log(\pi) \log(4\pi) \right) - \frac{1}{4} \log\beta \log(4\pi^2 \beta) \right). \label{P1 defn}
	\end{align}
	Another shift of the line of integration back to Re$(s)=c$, where $0<c<1$, and then the use of \eqref{P4} and \eqref{psi fe}  leads to
	\begin{align}
		J_2(\beta)= -2 \sum_{n=1}^{\infty}  \phi_0(n\beta)\log(n) + 2 P_{1}(\beta).	\label{J2 last}
		\end{align}
	From \eqref{sum integral R0} and \eqref{I A1 A2}, we get
	\begin{align}\label{new 1}
		\sqrt{\alpha} \sum_{n=1}^{\infty} \phi_1(n\alpha) = \sqrt{\beta} (A_1(\beta) + A_2(\beta))  + \sqrt{\alpha} R_0(\alpha).
	\end{align}
	Observe from \eqref{I integral}, \eqref{A1 defn} and \eqref{sum integral R0} that
	\begin{equation}\label{new2}
		\sqrt{\beta}A_1(\beta)=I(\beta)=\sqrt{\beta}	\sum_{n=1}^{\infty} \phi_1(n\beta) - \sqrt{\beta}R_0(\beta),
	\end{equation}
	where, the last equality follows from replacing $\alpha$ by $\beta$ in \eqref{sum integral R0}.  Substitute \eqref{new2} in \eqref{new 1} to obtain
	\begin{align}
		\sqrt{\alpha} \sum_{n=1}^{\infty} \phi_1(n\alpha) 
		=  \sqrt{\beta}	\sum_{n=1}^{\infty} \phi_1(n\beta) - \sqrt{\beta}R_0(\beta) + \sqrt{\beta} A_2 (\beta) + \sqrt{\alpha} R_0(\alpha). \label{almost there 1}
	\end{align}
	Interchanging $\alpha$ and $\beta$ in \eqref{almost there 1}, we have
	\begin{align}
		\sqrt{\beta}	\sum_{n=1}^{\infty} \phi_1(n\beta)= \sqrt{\alpha} 	\sum_{n=1}^{\infty} \phi_1(n\alpha) - \sqrt{\alpha}R_0(\alpha) + \sqrt{\alpha} A_2(\alpha) + \sqrt{\beta} R_0(\beta). \label{almost there 2}
	\end{align}
	Add \eqref{almost there 1} and \eqref{almost there 2} to get\footnote{This result can also be derived by starting with the definition of $A_2(\beta)$ given in \eqref{A2 defn}.}
	\begin{align}
		\sqrt{\alpha} A_2(\alpha) + \sqrt{\beta} A_2 (\beta) = 0. \label{A2 FE}
	\end{align}	
	A rearrangement of \eqref{almost there 1} yields
		\begin{align*}
		\sqrt{\alpha} 	\sum_{n=1}^{\infty} \phi_1(n\alpha) - \sqrt{\alpha} R_0(\alpha) - \frac{1}{2} \sqrt{\beta} A_2 (\beta) = \sqrt{\beta} 	\sum_{n=1}^{\infty} \phi_1(n\beta) - \sqrt{\beta}R_0(\beta) + \frac{1}{2} \sqrt{\beta} A_2 (\beta) .
	\end{align*}
	Invoking \eqref{A2 FE} on the left-hand side of the above equation, we get
		\begin{align*}
		\sqrt{\alpha} 	\sum_{n=1}^{\infty} \phi_1(n\alpha) - \sqrt{\alpha} R_0(\alpha) + \frac{1}{2} \sqrt{\alpha} A_2 (\alpha) = \sqrt{\beta} 	\sum_{n=1}^{\infty} \phi_1(n\beta) - \sqrt{\beta}R_0(\beta) + \frac{1}{2} \sqrt{\beta} A_2 (\beta),
	\end{align*}
	which is completely symmetric in $\alpha$ and $\beta$. Now employ \eqref{A2 J1 J2 J H}, \eqref{J1 last} and \eqref{J2 last} in the above equation, then use Ramanujan's formula \eqref{w1.26} and \eqref{Jequation} to arrive at
	\begin{align*}
		&\sqrt{\alpha} \left(	\sum_{n=1}^{\infty} \phi_1(n\alpha)  - \sum_{n=1}^{\infty} \log(n\alpha)\phi_0(n\alpha) -  R_0(\alpha) +  P_{1}(\alpha) - \log(\alpha) \left( \frac{\gamma - \log(2\pi \alpha)}{2\alpha} \right) +  \frac{1}{2} \mathscr{H}(\alpha)  \right) \notag \\
		=&\sqrt{\beta} \left(	\sum_{n=1}^{\infty} \phi_1(n\beta)  - \sum_{n=1}^{\infty} \log(n\beta)\phi_0(n\beta) -  R_0(\beta) +  P_{1}(\beta) - \log(\beta) \left( \frac{\gamma - \log(2\pi \beta)}{2\beta} \right) +  \frac{1}{2} \mathscr{H}(\beta)  \right). 
	\end{align*}
	Finally, use \eqref{R0 defn} and \eqref{P1 defn} and simplify to arrive at \eqref{main result 1}. This completes the proof.
\end{proof}
We are now ready to prove the main result of this paper, that is, Theorem \ref{log-twist}. 

\begin{proof}[Theorem \textup{\ref{log-twist}}][]
Subtract the corresponding sides of \eqref{main result 1} from those of\eqref{k=1} and simplify to arrive at the result.
\end{proof}	
The only reason the proof of this theorem looks short is because the majority of the reasoning which leads to this result is distributed in Lemma \ref{J lemma} and Theorem \ref{maintheorem}. 
	
\section{Different representations of $\mathscr{H}(x)$}\label{repH}
The integral $\mathscr{H}(x)$, defined in \eqref{H defn}, admits various interesting representations as integrals or series. Before stating them, we establish a lemma which is of independent interest.
	\begin{lemma}\label{lemma integral}
	For any $x>0$ and $1<d<2$,
	\begin{align}\label{lastsectionlemma}
		\int_{0}^{\infty} \frac{t \log(t)}{(1+t^2)(e^{xt}-1)} \, dt = \frac{\pi^2}{4} \cdot \frac{1}{2\pi i} \int_{(d)} \Gamma(w) \zeta(w)\cot \left( \frac{\pi w}{2}  \right) \textup{cosec} \left( \frac{\pi w}{2} \right)  x^{-w}   \, dw.
	\end{align}
\end{lemma}
\begin{proof}
	Let $a\in\mathbb{C}$ such that $0<\textup{Re}(a)<4$. From \cite[Equation (4.1), p. 182]{ober}, for $-\textup{Re}(a)<\textup{Re}(w)<2-\textup{Re}(a)$,
	\begin{align*}
		\int_{0}^{\infty} \frac{t^{w+a-1}}{1+t^2} \, dt = \frac{\pi}{2 \cos \left( \frac{\pi}{2} (w+a-1) \right)} =\frac{\pi}{2 \sin \left( \frac{\pi}{2} (w+a) \right)},
	\end{align*}
	whereas, for Re$(w)>1$ and $ x>0$,
	\begin{align*}
		\int_{0}^{\infty} \frac{t^{w-1}}{e^{xt}-1} \, dt = x^{-w} \Gamma(w) \zeta(w).
	\end{align*}
	Then by Parseval's formula \eqref{Par1}, for $\max(\textup{Re}(a)-1,1)<d=\textup{Re}(w)<\textup{Re}(a)+1$,
	\begin{align}
		\int_{0}^{\infty} \frac{t^a}{(1+t^2)(e^{xt}-1)} \, dt = \frac{1}{2\pi i} \int_{(d)} \frac{\pi x^{-w} \Gamma(w) \zeta(w)}{2 \sin \left( \frac{\pi}{2} (1-w+a) \right)}  \, dw = \frac{1}{2\pi i} \int_{(d)} \frac{\pi  \Gamma(w) \zeta(w)}{2 \cos \left( \frac{\pi}{2} (w-a) \right)} x^{-w} \, dw. \label{afterparseval}
	\end{align}
	We wish to differentiate extreme sides of \eqref{afterparseval} with respect to $a$. First consider the integral on the extreme left. Clearly, the integrand is a continuous function of both $t$ and $a$, and for a fixed $t$, it is an analytic function of $a$. Moreover, it is not difficult to see that the integral converges uniformly at both the limits in any compact subset of Re$(a)>0$. Thus, by \cite[p.~30-31, Theorem 2.3]{temme}, the integral on the extreme left of \eqref{afterparseval} is an analytic function of $a$ in Re$(a)>0$ and all of its derivatives can be obtained by differentiating under the integral sign. One can check that the same is true for the integral on the extreme right of \eqref{afterparseval}. Therefore, differentiating the extreme sides of \eqref{afterparseval} under the integral sign with respect to $a$, we have
	\begin{align}
		\int_{0}^{\infty} \frac{t^a \log(t)}{(1+t^2)(e^{xt}-1)} \, dt 
		&=-\frac{\pi^2}{4} \cdot \frac{1}{2\pi i} \int_{(d)} \Gamma(w) \zeta(w)\tan \left( \frac{\pi}{2} (w-a) \right) \sec \left( \frac{\pi}{2} (w-a) \right)  x^{-w}  \, dw. \label{afterparsevaldiff}
	\end{align}
	Now put $a=1$ in \eqref{afterparsevaldiff} to arrive at \eqref{lastsectionlemma}.
\end{proof}
We now give three representations for $\mathscr{H}(x)$ in the following theorem below.
	\begin{theorem}\label{repHtheorem}
	For $y>0$, let the function $\mathscr{H}(y)$ be defined in \eqref{H defn}. It admits alternate representations given below:
	\begin{align}
		\mathscr{H}(y) & =  -\frac{\pi^2}{4 y } + 4 \sum_{n=1}^{\infty} d(n) \int_{0}^{\infty} \frac{t \log(t) e^{-2\pi n y t}}{1+t^2}  \, dt \label{1identity}\\
		&= 2 \pi \int_{0}^{\infty} \left( \frac{1}{e^{2 \pi x} -1} - \frac{1}{2 \pi x}\right) \left(\log(x y)- \frac{1}{2} \left(\psi(i x y) + \psi(-i x y)\right)  \right)  \, dx \label{2identity}\\
		&=2 \pi \int_{0}^{\infty} \int_{0}^{\infty}\left( \frac{1}{e^{2 \pi x} -1} - \frac{1}{2 \pi x}\right)\left( \frac{1}{e^{t} -1} - \frac{1}{t}\right)\cos(xyt)\, dt\, dx.\label{2.5identity}
	\end{align}
\end{theorem}
\begin{proof}
Use the functional equation of $\zeta(w)$, stated in \eqref{asymmetric fe of zeta}, in the definition of $\mathscr{H}(y)$ given in \eqref{H defn}, and simplify to obtain for $0<c<1$,
	\begin{align*}
	\mathscr{H}(y) &= \pi^2 \cdot \frac{1}{2 \pi i} \int_{(c)} \Gamma(w) \zeta^2(w) \cot\left(\frac{\pi w }{2}\right)  \textup{cosec}\left(\frac{\pi w}{2}\right) (2\pi y)^{-w} \, dw \\
	&=  \pi^2 \cdot  \frac{1}{2 \pi i} \int_{(d)} \Gamma(w) \zeta^2(w) \cot\left(\frac{\pi w }{2}\right)  \textup{cosec}\left(\frac{\pi w}{2}\right) (2\pi y)^{-w} \, dw - \frac{\pi^2}{4y},
\end{align*}
where the last step follows from shifting the line of integration from Re$(s)=c$ to Re$(s)=d, 0<d<1$ and applying the residue theorem. Using the series definition of $\zeta(w)$, interchanging of the order of summation and integration because of absolute and uniform convergence, and then invoking \ref{lemma integral} in the next step, we see that
	\begin{align*}
	\mathscr{H}(y) &= -\frac{\pi^2}{4y} + \pi^2 \sum_{m=1}^{\infty} \frac{1}{2\pi i} \int_{(d)} \Gamma(w) \zeta(w) \cot\left(\frac{\pi w }{2}\right)  \textup{cosec}\left(\frac{\pi w}{2}\right) (2\pi m y)^{-w} \, dw \\
	&= -\frac{\pi^2}{4y} + 4  \sum_{m=1}^{\infty} \int_{0}^{\infty} \frac{t \log(t)}{(1+t^2)(e^{2 \pi m y t}-1)} \, dt \\
	&= -\frac{\pi^2}{4 y} + 4 \sum_{m=1}^{\infty} \int_{0}^{\infty} \frac{t \log(t)}{1+t^2} \sum_{k=1}^{\infty} e^{-2\pi mk y t} \, dt \\
	&= -\frac{\pi^2}{4 y } + 4 \sum_{n=1}^{\infty} d(n) \int_{0}^{\infty} \frac{t \log(t) e^{-2\pi n y t}}{1+t^2}  \, dt, 
\end{align*}
which gives the first representation \eqref{1identity}. To derive the second representation in \eqref{2identity}, we first replace $w$ by $1-w$ and simplify to obtain for $0<c=$Re$(w)<1$,
	\begin{align}
	\mathscr{H}(y)=\frac{2 \pi^2}{ y} \cdot \frac{1}{2 \pi i} \int_{(c)} \frac{ \zeta(w)  \zeta(1-w)}{4 \cos^2\left(\frac{\pi w}{2}\right)}  \left(\frac{1}{y}\right)^{-w} \, dw \label{doublesign}.
\end{align}
Now
	\begin{align*}
	\frac{1}{2 \pi i } \int_{(d)}  \Gamma(w) \zeta(w) x^{-w} \, dw =\frac{1}{e^{x}-1},
\end{align*}
which, upon using \eqref{asymmetric fe of zeta}, leads to
\begin{align*}
	\frac{1}{2 \pi i } \int_{(d)}  \frac{\zeta(1-w)}{2\cos\left( \frac{\pi w}{2} \right)} \left(\frac{x}{2\pi}\right)^{-w} \, dw  =\frac{1}{e^{x}-1}.
\end{align*}
Next, moving the line of integration from Re$(s)=d$ to Re$(s)=c$, where $0<c<1$, and then replacing $x$ by $2\pi x$, we have
	\begin{align}\label{P1}
	\frac{1}{2 \pi i } \int_{(c)} \frac{ \zeta(1-w)  }{2 \cos\left(\frac{\pi w}{2}\right)}  {x}^{-w} \, dw =  \frac{1}{e^{2\pi x} -1} -\frac{1}{2 \pi x}.
\end{align}
Now, from \cite[Lemma 4.1]{dgkm}\footnote{The argument in the proof of this lemma in \cite{dgkm} needs to be modified a bit in that instead of applying the commonly used Parseval's formula, that is, \eqref{Par2}, one has to resort to Vu Kim Tuan's extension of Parseval's formula; see \cite{vu} (especially Corollary 1) as well as \cite[pp.~15--17]{yakubovich}.}, 
we have, for $1<d'<2$ and Re$(x)>0$,
\begin{align*}
	\frac{1}{2 \pi i } \int_{(d')} \frac{ \Gamma(w)  }{\tan\left(\frac{\pi w}{2}\right)}  x^{-w} \, dw &=\frac{2}{\pi} \int_{0}^{\infty} \frac{t \cos(t)  }{x^2+t^2} \, dt. 
\end{align*}
Replace $x$ by $nx$ and sum both sides over $n \in \mathbb{N}$ to get
\begin{align}
	\frac{1}{2 \pi i } \int_{(d')} \frac{ \Gamma(w) \zeta(w) }{\tan\left(\frac{\pi w}{2}\right)}  x^{-w} \, dw &= \sum_{n=1}^{\infty}  \frac{2}{\pi} \int_{0}^{\infty} \frac{t \cos(t)}{n^2x^2+t^2} \, dt.  \label{mellin 2 half}
\end{align}
Use the functional equation of $\zeta(w)$ \eqref{asymmetric fe of zeta} in \eqref{mellin 2 half} and then replace $x$ by $2\pi x$ to obtain
	\begin{align*}
	\frac{1}{2 \pi i } \int_{(d')} \frac{ \zeta(1-w) }{2\sin\left(\frac{\pi w}{2}\right)} x^{-w} \, dw &= \frac{2}{\pi} \sum_{n=1}^{\infty}   \int_{0}^{\infty} \frac{t \cos(t)  }{4\pi^2 n^2x^2+t^2} \, dt.
\end{align*}
Move the line of integration to Re$(w)=c'$, where $0<c'<1$, and note that this process does not introduce any pole. Now employ the change of variable $w=1-s$ in the resulting integral so that for $0<c=\textup{Re}(s)<1$,
	\begin{align}
	\frac{1}{2 \pi i } \int_{(c)} \frac{ \zeta(s)  }{2 \cos\left(\frac{\pi s}{2}\right)}  x^{s} \, ds  &= \frac{2x}{\pi} \sum_{n=1}^{\infty} \int_{0}^{\infty} \frac{t \cos(t)  }{4\pi^2n^2 x^2+t^2} \, dt \\
	&= \frac{x}{\pi}  \left(\log(x )- \frac{1}{2} \left(\psi(i x ) + \psi(-i x )\right)  \right) \label{firstrep almost},
\end{align}
where, in the last step, we used  \cite[Theorem 2.2]{dgkm}
	\begin{align*}
	\sum_{m=1}^{\infty} \int_{0}^{\infty} \frac{t \cos(t)  }{t^2+m^2 u^2} \, dt = \frac{1}{2}  \left\{ \log\left(\frac{u}{2\pi}\right)- \frac{1}{2} \left(\psi\left(\frac{iu}{2\pi}\right) + \psi\left(-\frac{iu}{2\pi}\right)\right)  \right\},
\end{align*} 
with $u$ replaced by $2\pi x$. 
Replacing $x$ by $1/x$ in \eqref{firstrep almost} so that
\begin{align}
	\frac{1}{2 \pi i  } \int_{(c)} \frac{ \zeta(w)  }{2 \cos\left(\frac{\pi w}{2}\right)}  x^{-w} \, dw
	= \frac{1}{\pi x}  \left\{-\log(x)- \frac{1}{2} \left(\psi\left(\frac{i}{x}\right) + \psi\left(-\frac{i}{x}\right)\right)  \right\}.\label{P2}
\end{align}
We now apply Parseval's formula in the form \eqref{Par2} with $g(x)$ and $h(x)$ being the right-hand sides of \eqref{P1} and \eqref{P2} respectively so as to transform the integral on the right-hand side of \eqref{doublesign} to the integral in \eqref{2identity}.

We now derive the third representation for $\mathscr{H}(x)$ in \eqref{2.5identity} from the one in \eqref{2identity}. To that end, we note a result due to Wigert \cite[p.~203, Equation (3)]{wig0}, namely, for $a>0$,
\begin{align}\label{wigert equation}
	\int_{0}^{\infty}\left(\frac{1}{e^t-1}-\frac{1}{t}\right)\cos(at)\, dt = \log(a)-\frac{1}{2}\left(\psi(ia)+\psi(-ia)\right).
\end{align}
Now replace $\log(xy)-\frac{1}{2}\left(\psi(ixy)+\psi(-ixy)\right)$ occurring in \eqref{2identity} by the left-hand side of \eqref{wigert equation} with $a$ replaced by $xy$. This proves \eqref{2.5identity}.
\end{proof}

\section{Concluding remarks}
The Mellin transform of the auto-correlation function $A(x)$ defined in \eqref{auto-correlation function} is \cite[p.~1426]{darses-najnudel} $\displaystyle\frac{\pi \zeta(s)\zeta(1-s)}{\sin(\pi s)}$, which suggests that the function $\mathscr{H}$ defined in \eqref{H defn} should be related to the function $A(x)$. Another connection is evident through the integral representation of $\mathscr{H}(x)$ given in  \eqref{1identity} and the relation of $A(x)$ with the series $\sum_{n=1}^{\infty}d(n)e^{-2\pi nx}$ given in \cite[Section 2.2]{darses-najnudel}. \\
\begin{center}
\textbf{Acknowledgements}
\end{center}
The first author is supported by the Swarnajayanti Fellowship grant SB/SJF/2021-22/08 of ANRF (Government of India). The second author was an Institute Postdoc supported by the grant MIS/IITGN/R\&D/MATH/AD/2324/058 of IIT Gandhinagar. Both the authors sincerely thank the respective funding agencies for their support.

\end{document}